\documentclass[11pt,reqno]{amsart}
\usepackage[dvipdfm]{graphicx, color}
\usepackage{amssymb, amsmath,amsthm}
\newtheorem{thm}{{\bf Theorem}}[section]
\newtheorem{lem}[thm]{{\bf Lemma}}

\newtheorem{prop}[thm]{{\bf Proposition}}

\newtheorem{rem}[thm]{Remark}

\newtheorem{ques}[thm]{Question}
\newtheorem{conj}[thm]{Conjecture}
\numberwithin{equation}{section}

\begin{document}

%
%

\title{Entropy vs volume for pseudo-Anosov maps}

\author[E. Kin]{%
    Eiko Kin
}
\address{%
       Department of Mathematical and Computing Sciences \\
       Tokyo Institute of Technology \\
        Ohokayama, Meguro \\
        Tokyo 152-8552 Japan
}
\email{%
        kin@is.titech.ac.jp
}

\author[S. Kojima]{%
    Sadayoshi Kojima
}
\address{%
        Department of Mathematical and Computing Sciences \\
        Tokyo Institute of Technology \\
        Ohokayama, Meguro \\
        Tokyo 152-8552 Japan
}
\email{%
        sadayosi@is.titech.ac.jp
}

\author[M. Takasawa ]{%
    Mitsuhiko Takasawa
}
\address{%
        Department of Mathematical and Computing Sciences \\
        Tokyo Institute of Technology \\
        Ohokayama, Meguro \\
        Tokyo 152-8552 Japan
}
\email{%
        takasawa@is.titech.ac.jp
}



\thanks{%
The first author is partially supported by Grant-in-Aid for Young Scientists (B) (No. 20740031), 
MEXT, and 
the second author for Scientific Research (A) (No. 18204004), JSPS, Japan
} 

\begin{abstract} 
We will discuss theoretical and 
experimental results concerning comparison of 
entropy of pseudo-Anosov maps and volume of their mapping tori. 
Recent study of Weil-Petersson geometry of the Teichm\"uller space 
tells us that they admit linear inequalities for both sides under some 
bounded geometry condition.  
We construct a family of pseudo-Anosov maps which 
violates one side of inequalities under unbounded geometry 
setting, 
present an explicit bounding constant for a punctured torus,  
and provide several observations based on experiments. 
\end{abstract}

\maketitle

%
%

\section{Introduction} 
\label{section_introduction}

Let $\Sigma=\Sigma_{g,p}$ be an orientable surface of genus $g$ with $p$ punctures, 
$\mathcal{M}(\Sigma)$  the mapping class group of  $\Sigma$, 
and assume that  $3g - 3 + p \geq 1$.  
According to Thurston \cite{Thurston2}, 
the element of  $\mathcal{M}(\Sigma)$  is classified into three 
classes, 
namely, 
periodic, pseudo-Anosov and reducible.  
A pseudo-Anosov element  $\phi$  of  $\mathcal{M}(\Sigma)$  defines two
natural numerical invariants. 
One is the entropy  ${\rm ent}(\phi)$  which is the logarithm of a stretching factor of 
invariant foliation of  $\phi$   
(often called a dilatation of  $\phi$).   
The other is, 
thanks to the fibration theorem of Thurston \cite{Thurston3}, 
the hyperbolic volume  ${\rm vol}(\phi)$  of its mapping torus, 
\begin{equation*} 
	{\Bbb T}(\phi) =\Sigma \times [0,1]/ \sim
\end{equation*}  
where $\sim$ identifies $(x,1)$ with $(f(x),0)$  for 
some representative  $f$  of  $\phi$.   

Our study is motivated by experiments of the last author in his 2000 thesis 
\cite{Takasawa}  on comparison of  ${\rm ent}(\phi)$  and  ${\rm vol}(\phi)$.   
To see this, 
we let  $\mathcal{M}^{\mathrm{pA}}(\Sigma)$  be 
the set of pseudo-Anosov mapping classes of $\mathcal{M}(\Sigma)$ and put 
\begin{equation*} 
	\mathcal{E}(\Sigma) = \{(\mathrm{vol}(\phi), \mathrm{ent}(\phi)) 
	\ |\ \phi \in \mathcal{M}^{\mathrm{pA}}(\Sigma)\}.
\end{equation*} 
Figure \ref{fig_genus2} is the plot of  $\mathcal{E}(\Sigma_{2,0})$  for 
all pseudo-Anosov classes represented by words of length 
at most 7 with respect to 
the Lickorish generator.  


Although the samples may not be sufficiently many, 
it would be conceivable to predict that  ${\rm ent}(\phi)/{\rm vol}(\phi)$  
are bounded from both sides, 
namely, 
one may expect to have a constant  $C$  depending only on the topology of  
$\Sigma$  satisfying 
\begin{equation}\label{Eq:Bi-Lipschitz} 
	C^{-1} {\rm vol}(\phi) \leq {\rm ent}(\phi) \leq C {\rm vol}(\phi). 
\end{equation}
This expectation turns out to be true for bounded geometry case, 
which can be derived straightforwardly from recent works by 
Brock \cite{Brock}, Minsky \cite{Minsky} and Brock-Mazur-Minsky \cite{BMM}  
together with a comparison of Teichm\"uller and 
Weil-Petersson metrics on the Tehchm\"uller space.  
Moreover, 
the left inequality holds without assuming boundedness in geometry by 
a comparison of two metrics above.  
However, 
the theory does not say very much about 
accurate value of the constant  $C$.    

From computing viewpoints, 
it is rather easy to work with not closed but punctured disk cases, 
which have nice description in terms of braid data.  
Let  $D_n$  be an $n$-punctured disk.  
See Figure \ref{fig_br6} for more accurate plots for 
the case of  $D_6$.

\begin{figure}[htbp]
\begin{minipage}[htbp]{0.45\textwidth}
\begin{center}
\includegraphics[height=\textwidth, angle=-90]{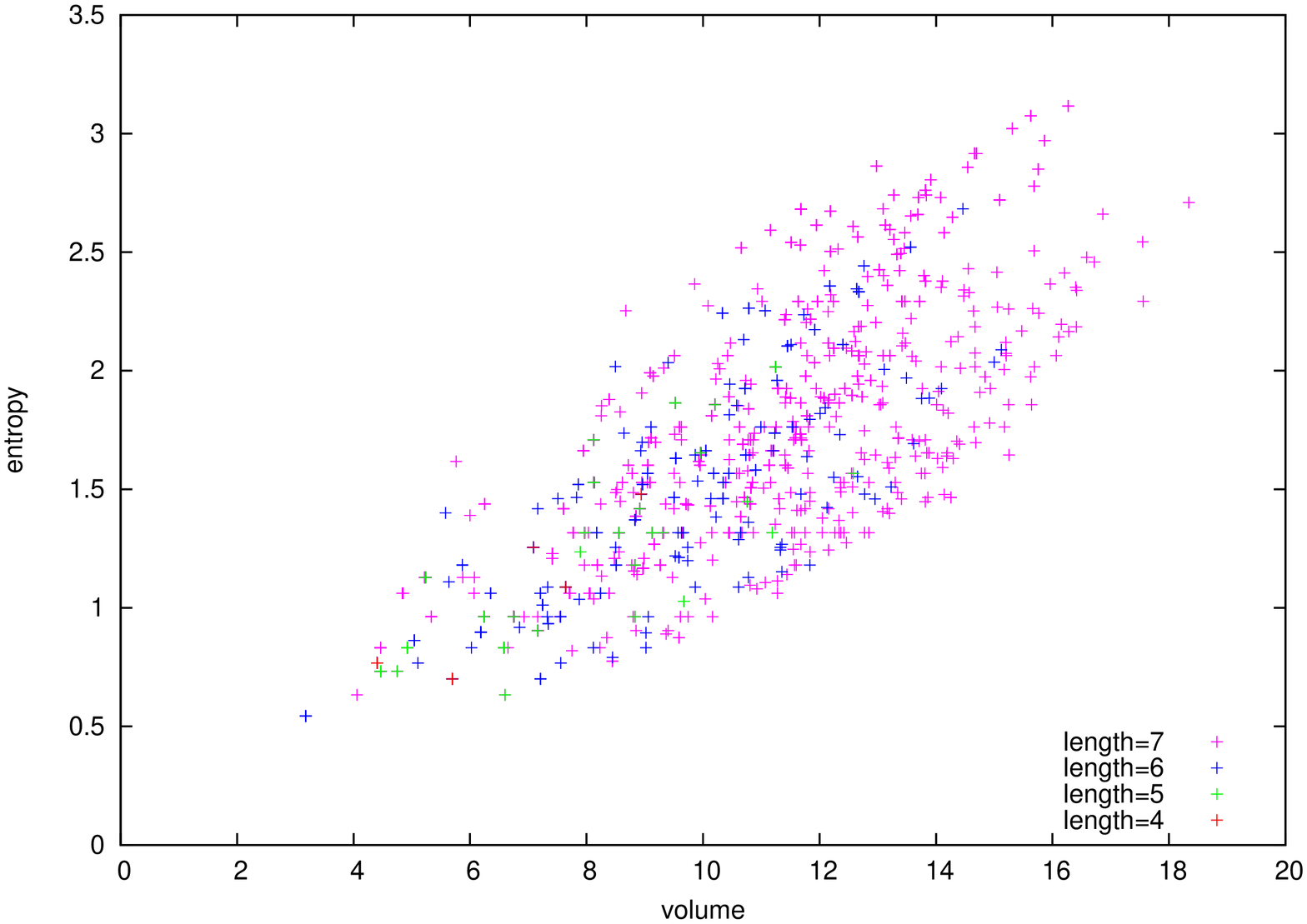}
\caption{Entropy vs Volume for  $\Sigma_{2,0}$.} 
\label{fig_genus2}
\end{center}
\end{minipage}
\hfill
\begin{minipage}[htbp]{0.45\textwidth}
\begin{center}
\includegraphics[height=\textwidth, angle=-90]{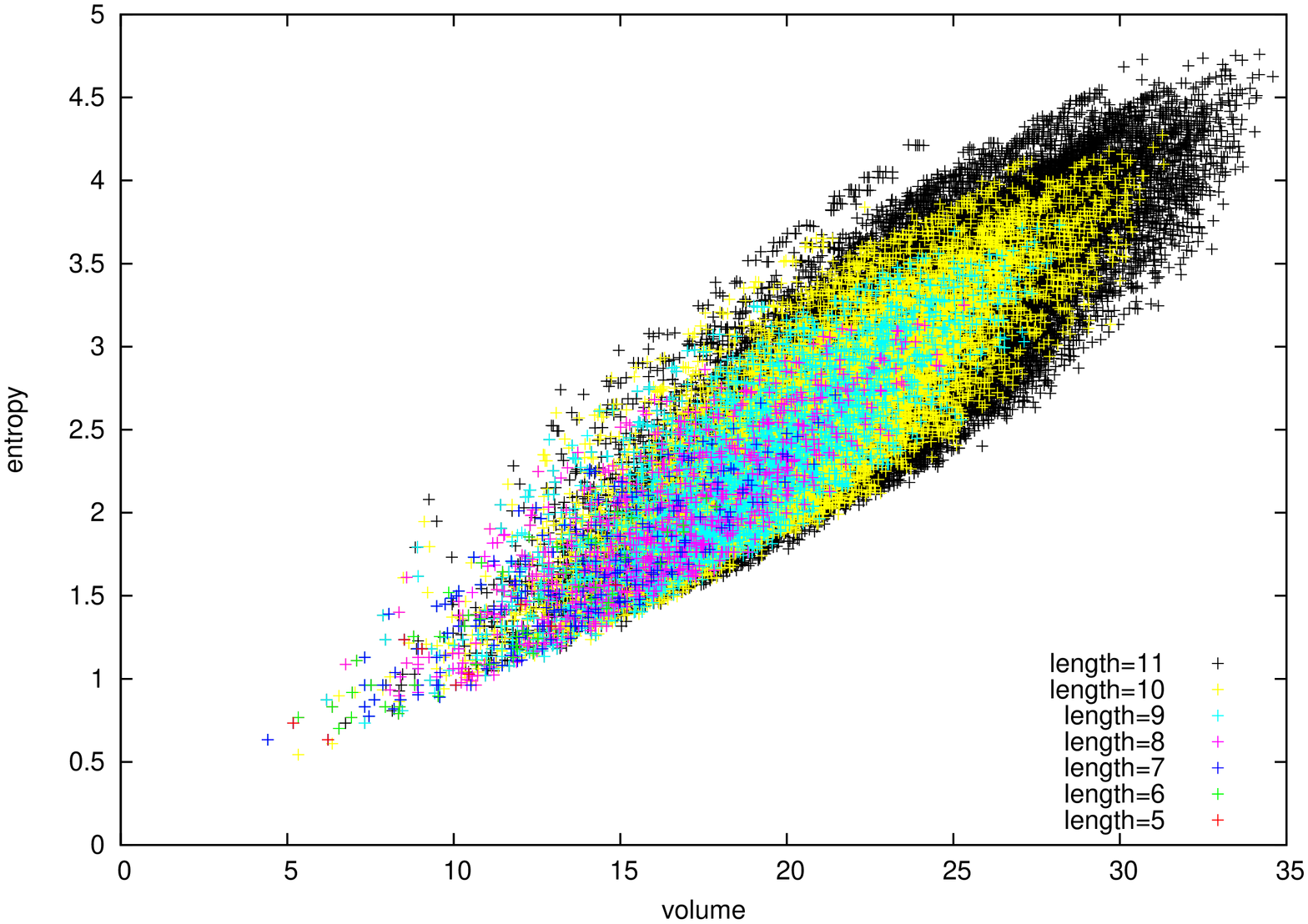}
\caption{Entropy vs Volume for  $D_6$.} 
\label{fig_br6_small}
\end{center}
\end{minipage}
\end{figure} 

The purpose of this paper is to present two theoretical results and 
to provide some observations based on our experiments for 
pseudo-Anosov maps on punctured disks.  

More precisely, 
we explicitly construct  a family of pseudo-Anosov maps  $\phi_k$ 
such that  $\mathrm{ent}(\phi_k)$  goes to infinity as 
$n$ goes to infinity while $\mathrm{vol}(\phi_k)$ remains bounded from above 
in {\bf Theorem \ref{thm_entdiverge}},    
which violates the inequality of the right hand side of  (\ref{Eq:Bi-Lipschitz}),  
where geometry in this family is unbounded.  
We also give an explicit constant for  $\Sigma_{1,1}$.    
Namely, 
we prove in {\bf Theorem \ref{thm_punctured-torus}}  that 
for each $\phi \in  \mathcal{M}^{\mathrm{pA}}(\Sigma_{1,1})$, we have 
\begin{equation*} 
	\frac{\rm{ent}(\phi)}{\mathrm{vol}(\phi)} > 
	\frac{ \log (\frac{3+  \sqrt{5}}{2})}{2  v_8}\approx 0.1313,
\end{equation*} 
where $v_8 \approx 3.6638$ is the volume of a regular ideal octahedron. 
This bound is not best possible unfortunately.
However, 
if we restrict our attention  to mapping classes of {\it block length} $1$, 
we obtain the best possible lower bound 
\begin{equation*} 
	\frac{ \log (\frac{3+  \sqrt{5}}{2})}{2  v_3} \approx 0.4741
\end{equation*} 
where $v_3 \approx 1.0149$ is the volume of  
a regular ideal tetrahedron in {\bf Proposition~\ref{prop_block-length-1}}. 
The organization of this paper is as follows.
After some preliminaries in the next section,
we present what can be known from
recent results and prove Theorem \ref{thm_Bi-Lipschitz} in section 3.
We describe our experiments and observations derived from them
in section 4, and then discuss sharp bounds in section 5.

\bigskip

\noindent
Acknowledgments: 
We would like to thank Kazuhiro Ichihara, Masaharu Ishikawa 
and Kenneth J.~Shackleton  for many discussions and conversations. 
Our thanks also go to Shigenori Matsumoto for helpful comments.

%
%

\section{Preliminaries} 
\label{section_preliminaries}


\subsection{Perron-Frobenius theorem}
\label{subsection_PF}

Let $M=(m_{ij}) $ and $N=(n_{ij})$ be matrices with the same size. 
We shall write  $M \ge N$ (resp. $M >N$) 
whenever  $m_{ij} \ge n_{ij}$ (resp. $m_{ij} > n_{ij}$) for each $ij$. 
We say that $M$ is {\it positive} (resp. {\it non-negative}) 
if $M >{\bf 0}$ (resp. $M \ge {\bf 0}$), where ${\bf 0}$ is the zero matrix.  

For a square and non-negative matrix $T$,  let $\lambda(T)$ be its spectral radius, 
that is the maximal absolute value of eigenvalues of $T$.  
We say that  $T$ is {\it irreducible} if for every pair of indices $i$ and $j$, there exists an integer $k=k_{ij}>0$ such that the $(i,j)$ entry of 
$T^k$ is strictly positive. 
The matrix  $T$ is {\it primitive} if there exists an integer $k >0$ such that the matrix $T^k $ is positive.  
By definition, a primitive matrix is irreducible.  
A primitive matrix $T$ is {\it Perron-Frobenius}
 if $T$ is an integral matrix. 
The following theorem is commonly referred  to as  the Perron-Frobenius theorem, 
see for example Theorem 1.1 in \cite{Seneta}.  

\begin{thm}[Perron-Frobenius]
\label{thm_PFtheorem}
Let $T$ be a primitive matrix. 
Then, there exists an eigenvalue $\lambda>0$ of $T$ such that 
\begin{itemize}
\item[(1)]  $\lambda$ has strictly positive left and right eigenvectors ${\bf \widehat{x}}$ and ${\bf y}$ respectively, and 
\item[(2)]  $\lambda>|\lambda'|$ for any eigenvalue $\lambda' \ne \lambda$ of $T$. 
\end{itemize}
If $T \ge B \ge {\bf 0}$ and $\beta$ is an eigenvalue of $B$, then $\lambda \ge |\beta|$. 
Moreover $\lambda=  |\beta|$ implies $T= B$. 
\end{thm}


\subsection{Pseudo-Anosov mapping classes} 

The {\it mapping class group} $\mathcal{M}(\Sigma)$ is the group of isotopy classes of orientation preserving homeomorphisms of $\Sigma$, 
where the group operation is induced by composition of homeomorphisms. 
An element of the mapping class group  is called a {\it mapping class}. 

A homeomorphism $\Phi: \Sigma \rightarrow \Sigma$ is {\it pseudo-Anosov}  
 if  there exists a constant $\lambda= \lambda(\Phi)>1$ called the {\it dilatation of} $\Phi$  
 and there exists a pair of transverse measured foliations $\mathcal{F}^s$ and $\mathcal{F}^u$ such that 
\begin{equation*} 
	\Phi(\mathcal{F}^s)= \frac{1}{\lambda} \mathcal{F}^s \ \mbox{and}\  
		\Phi(\mathcal{F}^u)= \lambda \mathcal{F}^u. 
\end{equation*} 
 A mapping class  which contains a pseudo-Anosov homeomorphism is called {\it pseudo-Anosov}. 
 We define the dilatation of a pseudo-Anosov  mapping class $\phi$, denoted by $\lambda(\phi)$, to be the dilatation of a pseudo-Anosov homeomorphism of $\phi$. 
Fixing $\Sigma $, the dilatation $\lambda(\phi)$ for $\phi \in \mathcal{M}^{\mathrm{pA}}(\Sigma)$ is known to be an algebraic integer with a 
bound on its degree depending only on $\Sigma$. 
Thus the set 
\begin{equation*} 
	\mathrm{Dil}(\Sigma) = \{\lambda(\phi)>1\ |\ \phi \in \mathcal{M}^{\mathrm{pA}}(\Sigma)\}
\end{equation*} 
is discrete and in particular achieves its infimum  $\lambda(\Sigma)$. 
 
The ({\it topological}) {\it entropy} $\mathrm{ent}(f)$ is a measure of the complexity of a continuous self-map $f$ on a compact metric space, see for instance  \cite{Walters}. 
For a pseudo-Anosov homeomorphism $\Phi$, 
the equality $\mathrm{ent}(\Phi)= \log (\lambda(\Phi))$ holds \cite{FLP} and 
$\mathrm{ent}(\Phi)$ attains the minimal entropy among all  homeomorphisms  $f$ 
which are isotopic to $\Phi$. 

Choosing a representative $f: \Sigma \rightarrow \Sigma$ of $\phi \in \mathcal{M}(\Sigma)$, 
we form the mapping torus 
$${\Bbb T}(\phi) =\Sigma \times [0,1]/ \sim ,$$ 
where $\sim$ identifies $(x,0)$ with $(f(x),1)$. 
Then $\phi $ is pseudo-Anosov  if and only if ${\Bbb T}({\phi})$ admits 
a complete hyperbolic structure of finite volume \cite{Thurston3, OK}. 
Since such a structure is unique up to isometry, 
it makes sense to speak of  the volume $\mathrm{vol}(\phi)$ of $\phi$,  
the hyperbolic volume of ${\Bbb T}(\phi)$.


\subsection{Generating sets of mapping class groups}

Let $G$ be a group with a symmetric generating set $\mathcal{G}$ 
so that if  $h \in \mathcal{G}$ then $h^{-1} \in \mathcal{G}$. 
The {\it word length} of $h$ {\it relative to} $\mathcal{G}$ is defined by 
$\min\{k\ |\ h = h_1 h_2 \cdots h_k, \ h_i \in \mathcal{G}\}$.

We introduce a generating set of 
the mapping class group   $\mathcal{M}(D_n)$  by 
\begin{equation*} 
	\mathcal{G}(D_n) =  \{\hat{t}_{c_1}^{\pm 1}, \cdots, \hat{t}_{c_{n-1}}^{\pm 1}\}, 
\end{equation*} 
where $\hat{t}_{c_i}$ denotes the mapping class which represents the positive half twist about the arc $c_i$ from the $i^{\mathrm{th}}$ puncture to the $(i+1)^{\mathrm{st}}$ (Figure~\ref{fig_halftwist}).

\begin{figure}[htbp]
\begin{center}
\includegraphics[width=5.5in]{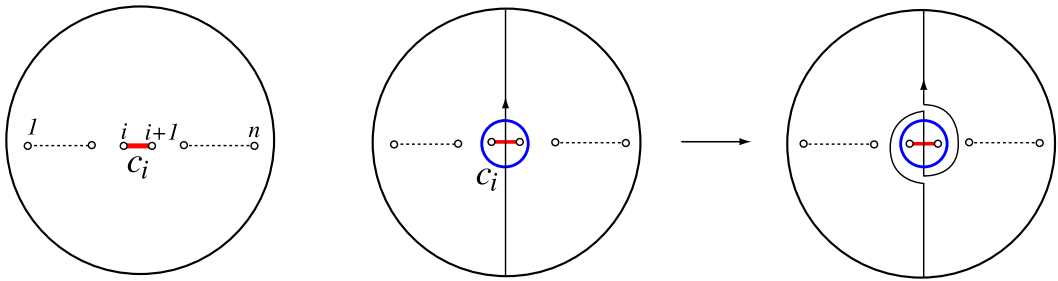}
\caption{(left) arc $c_i$, (right) positive half twist $\hat{t}_{c_i}$.} 
\label{fig_halftwist}
\end{center}
\end{figure}

The  $n$-braid group $B_n$ and the mapping class group $\mathcal{M}(D_n)$ are related by the surjective  homomorphism 
\begin{eqnarray*}
\Gamma: B_n &\rightarrow& \mathcal{M}(D_n)
\\
\sigma_i &\mapsto& \hat{t}_{c_i}
\end{eqnarray*}
where $\sigma_i$ for $i \in \{1, \cdots, n-1\}$ is the  Artin generator (Figure~\ref{fig_artin}(left)).  
The kernel of $\Gamma$ is the center of $B_n$ which is generated by a full twist braid 
$(\sigma_1 \sigma_2 \cdots \sigma_{n-1})^n$. 
Note that $\mathcal{M}(D_n)$ is isomorphic to a subgroup of $\mathcal{M}(\Sigma_{0,n+1})$ by replacing the boundary of $D_n$ with the $(n+1)^{\mathrm{st}}$  puncture. 
In the rest of the paper we regard a mapping class of $\mathcal{M}(D_n)$ to be a mapping class of $\mathcal{M}(\Sigma_{0,n+1})$ 
fixing the $(n+1)^{\mathrm{st}}$  puncture.

We say that a braid $b \in B_n$ is {\it pseudo-Anosov} if $\Gamma(b) \in \mathcal{M}(D_n)$ is pseudo-Anosov, and when this is the  case,  
$\mathrm{vol}(\Gamma(b))$ equals the hyperbolic volume of  the link complement $S^3 \setminus \overline{b}$ in the $3$--sphere $S^3$, where 
$\overline{b}$ is the braided link of $b$ which is a union of the closed braid of $b$ and the braid axis (Figure~\ref{fig_artin}(right)). 
Hereafter we represent a mapping class of  $ \mathcal{M}(D_n)$ by a braid and we denote $\Gamma(b) \in \mathcal{M}(D_n)$ by $b$.

\begin{figure}[htbp]
\begin{center}
\includegraphics[width=3.7in]{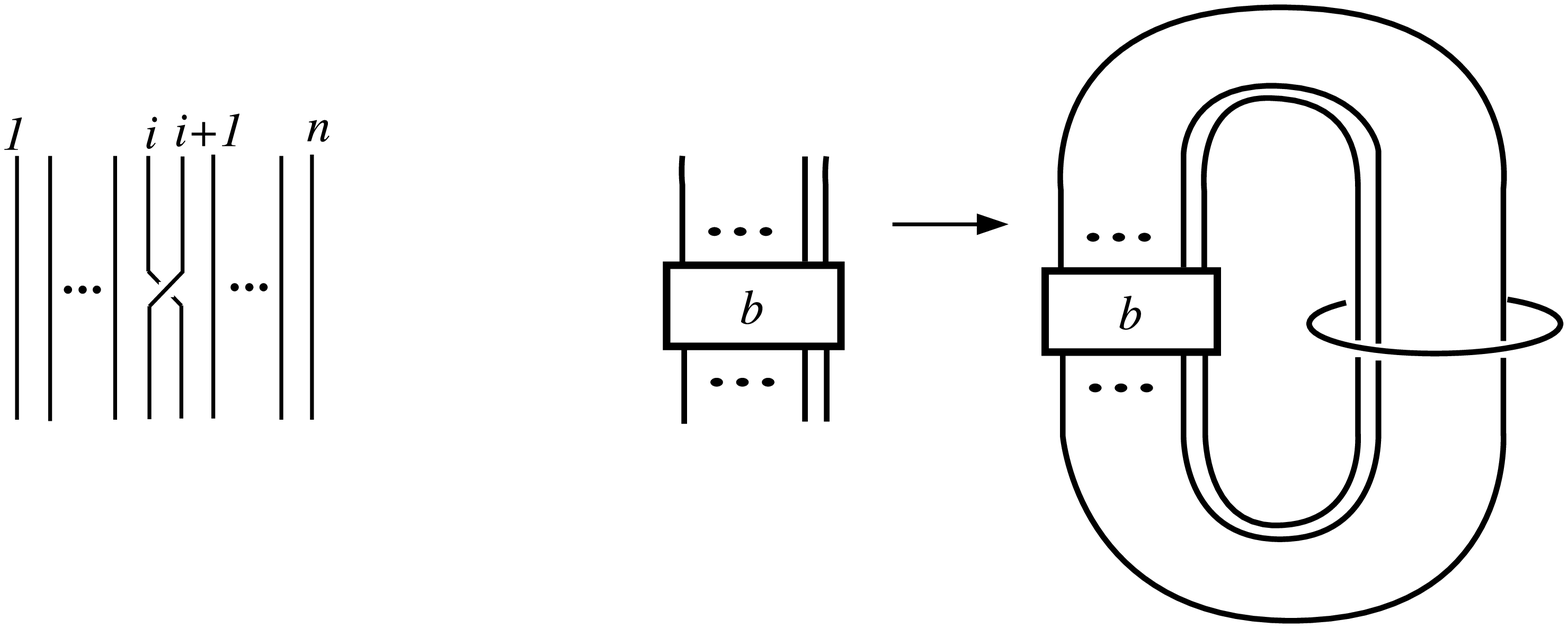}
\caption{(left) generator $\sigma_i$, (right) link $\overline{b}$ obtained from a braid $b$.}
\label{fig_artin}
\end{center}
\end{figure}

Given a simple closed curve $\alpha$ on $\Sigma$, let $t_{\alpha} $ be a mapping class which represents the positive Dehn twist about $\alpha$. 
Then $\mathcal{G}(\Sigma_{1,1})= \{t_a^{\pm 1}, t_b^{\pm 1}\}$ is a generating set for $\mathcal{M}(\Sigma_{1,1})$, 
where $a$ and $b$ are  the meridian and the longitude of a once-punctured torus.

%
%

\section{Theory for entropy vs volume}  
\label{section_theory}


\subsection{Linear inequalities in bounded geometry}

We here briefly review what could be known about 
entropy vs volume for  $\Sigma$  from 
the existing theory.  

To see this more precisely, 
let us introduce two norms for a pseudo-Anosov  $\phi$  with 
respect to metrics on the Teichm\"uller space.  
Let  $|| \phi ||_*$  be the translation distance of  $\phi$  with respect 
to the metric  $*$,   
where  $*$  is the Teichm\"uller metric which we 
denote by  $*=T$  or the Weil-Petersson metric by  $* = WP$.  
This is nothing but the minimal distance of the orbit of the action of  
$\phi$  on the Teichm\"uller space with respect to corresponding metrics.  
Notice here that 
\begin{equation*}
	{\rm ent}(\phi) = || \phi ||_T 
\end{equation*} 

Starting point is a seminal result by Brock  \cite{Brock},  
which shows that there is a constant  $D$  depending only 
the topology of  $\Sigma$  so that
\begin{equation*}
	D^{-1} \, {\rm vol}(\phi) \leq || \phi ||_{WP} \leq D \, {\rm vol}(\phi)
\end{equation*}
for any pseudo-Anosov  $\phi$  on  $\Sigma$.  
Thus we would like to compare two norms.  
The Teichm\"uller distance is originally studied as a distance derived from
quasi-conformal maps between two Riemann surfaces, 
and Linch  \cite{Linch}  succeeded to obtain comparison of two metrics directly.  
The modern theory introduces an infinitesimal interpretation of 
Teichm\"uller metric, 
see for instance \cite{GL}, 
and it leads us to Linch's inequality 
\begin{equation*}
	|| \phi ||_{WP} \leq - 2\pi \chi(\Sigma) || \phi ||_T 
\end{equation*}
just based on the Cauchy-Schwarz inequality between 
infinitesimal forms of two metrics, 
see for instance in  \cite{Royden}.  
This together with Brock's inequality immediately implies 
the first half of Theorem \ref{thm_Bi-Lipschitz} which will be stated below.  

The inequality of the right hand side in  (\ref{Eq:Bi-Lipschitz}) dose not hold in general 
as we will see explicitly in Theorem \ref{thm_entdiverge}.  
However we would like to have control 
under some bounding condition on the geometry of  $\phi$. 
The deep analysis carried out for Teichm\"uller metric by Minsky \cite{Minsky}  
and for Weil-Petersson metric by Brock-Mazur-Minsky  \cite{BMM}  are the ones 
we are looking for.  
Just one of conclusions of their works could be stated as follows for our purpose.  

\begin{thm}[Minsky \cite{Minsky}, Brock-Mazur-Minsky \cite{BMM}]
For any  $\varepsilon > 0$, 
there exists  $\delta > 0$  such that 
both Teichm\"uller and Weil-Petersson geodesics 
invariant by the action of a pseudo-Anosov  $\phi$  has no intersection 
with the subset of the Teichm\"uller space 
consisting of hyperbolic surfaces with closed 
geodesic of length  $\leq \delta$  if  
$\Bbb{T}(\phi)$  contains no closed geodesics of length  $\leq \varepsilon$.  
\end{thm}

Since the part of the Teichm\"uller space by thick surfaces appeared above  
is invariant by the action of the mapping class group, 
and moreover the quotient is compact by Mumford, 
it leads us to show the other direction of 
the inequality. 
We just state the the result one could straightforwardly 
conclude from the works by Minsky  \cite{Minsky}  and Brock-Mazur-Minsky  \cite{BMM}  by 

\begin{thm}[Corollary to \cite{Minsky, BMM}]\label{thm_Bi-Lipschitz}
There exists a constant  $B = B(\Sigma)$  depending only on 
the topology of   $\Sigma$  such that 
\begin{equation*}
	B \, {\rm vol}(\phi) \leq {\rm ent}(\phi)
\end{equation*} 
holds for any pseudo-Anosov  $\phi$  on  $\Sigma$.  
Furthermore, 
for any  $\varepsilon > 0$, 
there exists a constant  $C = C(\varepsilon, \Sigma) > 1$  depending only on  $\varepsilon$  and 
the topology of  $\Sigma$  such that the inequality 
\begin{equation*}
	{\rm ent}(\phi) \leq C \, {\rm vol}(\phi)  
\end{equation*} 
holds for any pseudo-Anosov  $\phi$  on  $\Sigma$  whose  
mapping torus  $\Bbb{T}(\phi)$  has no closed geodesics of length  $\leq \varepsilon$.  
\end{thm}


\subsection{Entropy wins Volume in unbounded geometry} 

The following  asserts that the inequality of 
the right hand side of  (\ref{Eq:Bi-Lipschitz}) does not hold 
without bounded geometry condition.

\begin{thm}
\label{thm_entdiverge}
  There exists a  sequence of $\phi_k \in \mathcal{M}^{\mathrm{pA}}(\Sigma)$ 
such that $\displaystyle{\frac{{\rm ent}(\phi_k)}{{\rm vol}(\phi_k)}}$ is unbounded.  
\end{thm}

We first review a construction of pseudo-Anosov mapping classes 
by Penner \cite{Penner} and how to compute their  dilatation. 
We follow the notation in \cite{Leininger}.  
Let $\mathcal{S}(\Sigma)  $ be the set of isotopy classes of essential simple closed curves on $\Sigma$. 
Let $\mathcal{S}'(\Sigma)$ denote the set of isotopy classes of essential, closed $1$-manifolds embedded on $\Sigma$. 
We refer to the components of $A \in \mathcal{S}'(\Sigma)$ as elements of $\mathcal{S}(\Sigma)$ and we write 
$A= a_1 \cup \cdots \cup a_m $ for $a_i \in \mathcal{S}(\Sigma)$.  
For simplicity we identify elements of $\mathcal{S}'(\Sigma)$ with $1$-manifolds on $\Sigma$, and 
we assume that elements $A,B \in \mathcal{S}'(\Sigma)$ meet transversely with minimal intersection number. 
We say that $A \cup B$ {\it fills} $\Sigma$ if each component of $\Sigma \setminus (A \cup B)$ is topologically a disk  or a disk with a puncture. 

For $A = a_1  \cup \cdots \cup a_m \in  \mathcal{S'}(\Sigma)$ and $B= b_1 \cup \cdots \cup b_n \in  \mathcal{S'}(\Sigma)$, 
suppose that $A \cup B$ fills $\Sigma$. 
Let $\mathcal{R}(A,B)$ be the free semigroup of $\mathcal{M}(\Sigma)$ generated by the Dehn twists 
$$\{t_{a_i}^{+1}\ |\ a_i \in A\} \cup \{t_{b_j}^{-1}\ |\ b_j \in B\}.$$ 
Penner  shows that $\phi \in \mathcal{R}(A,B)$ is pseudo-Anosov if each $t_{a_i}^{+1}$ for $i \in \{ 1, \cdots, m\}$ and each 
$t_{b_j}^{-1}$ for $j \in \{1, \cdots, n\}$ occurs at least once in $\phi$ \cite{Penner}.

We now construct a (bigon) train track $\tau$ obtained from $A \cup B$ by a deformation of each intersection $a_i \cap b_j$ 
as in Figure~\ref{fig_train_track_Penner}.

\begin{figure}[htbp]
\begin{center}
\includegraphics[width=2.5in]{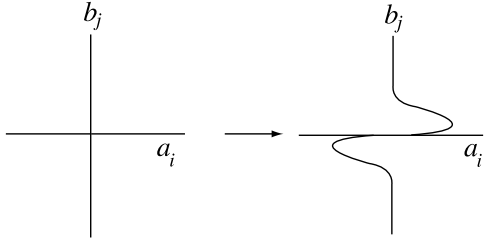}
\caption{train track obtained from a deformation.} 
\label{fig_train_track_Penner}
\end{center}
\end{figure} 

In a neighborhood $N(c)$ of  a simple closed curve $c \in A \cup B$, 
we take a push-off $c' $ of $c$ in either side  in $N(c)$ (Figure~\ref{fig_push-off-traintrack}(left)).  
We represent $\phi$ using $t_{a'_i}$ and $t_{b'_j}^{-1}$ rather than $t_{a_i}$ and $t_{b_j}^{-1}$. 
Note that  $\tau$ carries each $t_{a'_i}(\tau)$ (Figure~\ref{fig_push-off-traintrack}) and  each $t_{b'_j}^{-1}(\tau)$. 
Thus $\tau$ carries each $\phi(\tau)$  for each $\phi \in  \mathcal{R}(A,B)$. 
This implies that each $\phi $ induces a graph map $\phi_*: \tau \rightarrow \tau$. 
The non-negative integral matrix $M= (m_{tu})$, called the {\it incident matrix} for $\phi_*$, can be defined as follows:  
each $m_{tu}$ equals the number of the times  the image of the $u^{\mathrm{th}}$ edge of $\tau$ under $\phi_*$ passes through the $t^{\mathrm{th}}$ edge. 
Then $\lambda(\phi)$ for a pseudo-Anosov mapping class $\phi \in  \mathcal{R}(A,B)$ is equal to  the spectral radius of $M$ \cite{Penner}. 
(Here $M$ is in fact  a Perron-Frobenius matrix, and hence $\lambda(\phi)$ equals the largest eigenvalue of $M$ strictly greater than $1$.) 
\medskip

\begin{figure}[htbp]
\begin{center}
\includegraphics[width=3.3in]{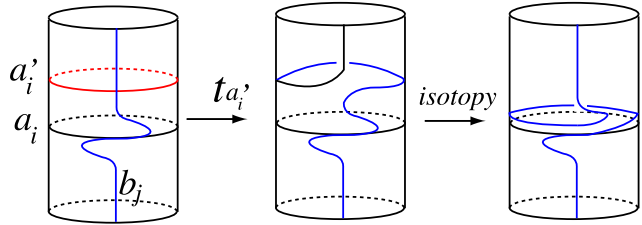}
\caption{image of $\tau$ under $t_{a_i'}$. } 
\label{fig_push-off-traintrack}
\end{center}
\end{figure} 

\begin{proof}[Proof of Theorem~\ref{thm_entdiverge}]
For the proof, it suffices to show that  there exists a  sequence 
of $\phi_k \in \mathcal{M}^{\mathrm{pA}}(\Sigma)$ 
such that $\lim_{k \to \infty} \mathrm{ent}(\phi_k) = \infty$ and $\mathrm{vol}(\phi_k) $ 
remains bounded above. 
Let $\phi = t_{a_1} t_{a_2} \cdots t_{a_m} t_{b_1}^{-1} t_{b_2}^{-1} \cdots t_{b_n}^{-1}$ and consider a family of pseudo-Anosov mapping classes 
$\phi_k  = t_{a_1}^k \phi$ for each $k >0$. 
We show that this is a desired family. 
Let $\hat{a}$ denote the knot $a_1 \times \{1/2\}$ in ${\Bbb T}(\phi)$ and 
let $W$ denote the closure of ${\Bbb T}(\phi) \setminus N(\hat{a})$, where 
$N(\hat{a})$ is a regular neighborhood of $\hat{a}$. 
Then $W(m/n)$ denotes the manifold obtained 
from ${\Bbb T}(\phi)$ by an $(m,n)$-Dehn filling on $W$, 
relative to a suitable choice of the longitude in $\partial N(\hat{a})$. 
We notice that  $W(1/k)$  is homeomorphic to ${\Bbb T}(t_{a_1}^k \phi)$. 
Hence by \cite[Proposition~6.5.2]{Thurston}, we have 
\begin{equation*} 
	\frac{\mathrm{vol}({\Bbb T}(t_{a_1}^k \phi))}{v_3} \le \| [W, \partial W] \|,
\end{equation*} 
where 
$ [W, \partial W] $ is the relative fundamental class, 
$\| \cdot \|$ is  the Gromov norm for a homology class $\cdot \in  H_*(W, \partial W)$. 
The volume of $\phi_k $ is thus bounded above. 

The proof is completed once  one shows that 
$\lambda(\phi_{k'}) > \lambda(\phi_{k})$ for $k' > k$,  
since the set $\mathrm{Dil}(\Sigma)$ is discrete. 
Consider a vertex $v$ of $\tau$. 
Let $\beta_t$ be an edge of $\tau$ emanating from $v$ such that $\beta_t$ is 
contained in some $b_j$ and 
$\beta_t$ intersects with the push-off $a_1'$. 
Let $\alpha_u$ be any edge of $\tau$ contained in $a_1$ (Figure~\ref{fig_vertex_nbhd}). 
The incident matrix for the mapping class $t_{a_1'}^k$ is of the form 
\begin{equation*} 
	N_k = I + P_k,
\end{equation*}
where $I$ is the identity matrix and $P_k$ is a non-negative integral matrix such that the 
$(u,t)$ entry equals $k$. 
If $k' > k$ we have $N_{k'} > N_k$, and in particular $N_{k'}  M > N_k M$, 
where $M$ is the incident matrix for $\phi$. 
Note that  $N_k M$ is the incident matrix for $\phi_k =t_{a_1}^k \phi$, and 
we have $\lambda(\phi_{k'}) > \lambda(\phi_k)$ since $N_{k'} M > N_k M$. 
\end{proof}

\begin{figure}[htbp]
\begin{center}
\includegraphics[width=2.3in]{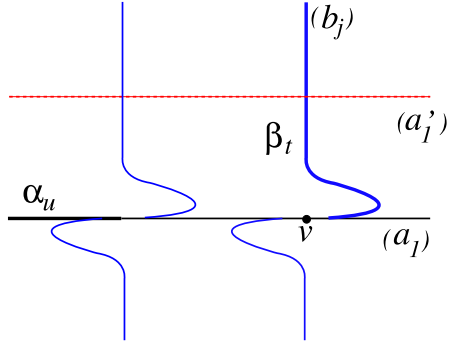}
\caption{$\beta_t$ and $\alpha_u$.} 
\label{fig_vertex_nbhd}
\end{center}
\end{figure}

%
%

\section{Experiments and Observations}\label{section_EandO}


\subsection{Experiments} 
\label{subsection_experiments}

For a computation of  the braid dilatation, we use a program by T.~Hall \cite{Hall}. 
For a computation of  the volume of links in the 3-sphere $S^3$, 
we use the program ``SnapPea'' by J.~Weeks \cite {Weeks}.
Here we exhibit the computation for $\Sigma \in \{D_3,D_4,D_5,D_6 \}$.  

We have 
\begin{equation*} 
	(\mathrm{vol}(\phi^m), \mathrm{ent}(\phi^m)) 
	= (m \, \mathrm{vol}(\phi), m \, \mathrm{ent}(\phi))
\end{equation*}  
and hence for any mapping class $\phi \in  \mathcal{M}^{\mathrm{pA}}(\Sigma) $, 
the line in ${\Bbb R}^+ \times {\Bbb R}^+$ of the slope 
$\mathrm{ent}(\phi)/\mathrm{vol}(\phi)$ passing through 
the origin must intersect $\mathcal{E}(\Sigma) $ in infinitely many points. 
Let 
\begin{equation*} 
	\mathcal{E}_k(\Sigma) = \{(\mathrm{vol}(\phi), \mathrm{ent}(\phi)) 
	\ |\ \phi \in \mathcal{M}^{\mathrm{pA}}(\Sigma) \ \mbox{up\ to\ word\ length\ }k\}. 
\end{equation*} 
The following is the plots of $\mathcal{E}_{15}(D_3)$, 
$\mathcal{E}_{12}(D_4)$, $\mathcal{E}_{10}(D_5)$ and
$\mathcal{E}_{11}(D_6)$
(Figures \ref{fig_br3} -- \ref{fig_br6}).

\begin{figure}[htbp]
\begin{center}
\includegraphics[height=0.9\textwidth, angle=-90]{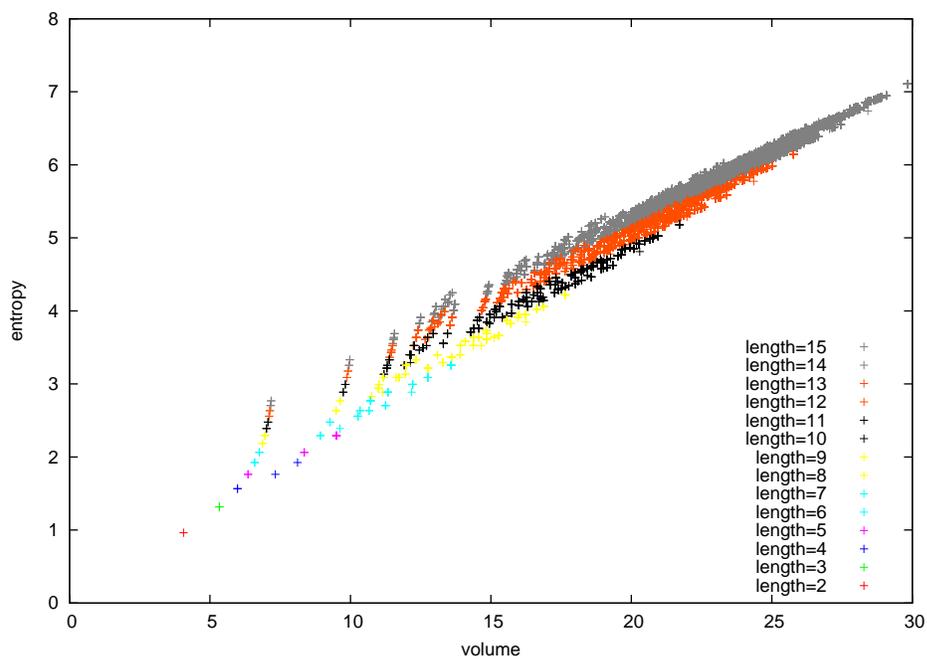}
\caption{$\mathcal{E}_{15}(D_3)$.} 
\label{fig_br3}
\end{center}
\end{figure} 

\begin{figure}[htbp]
\begin{center} 
\includegraphics[height=0.9\textwidth, angle=-90]{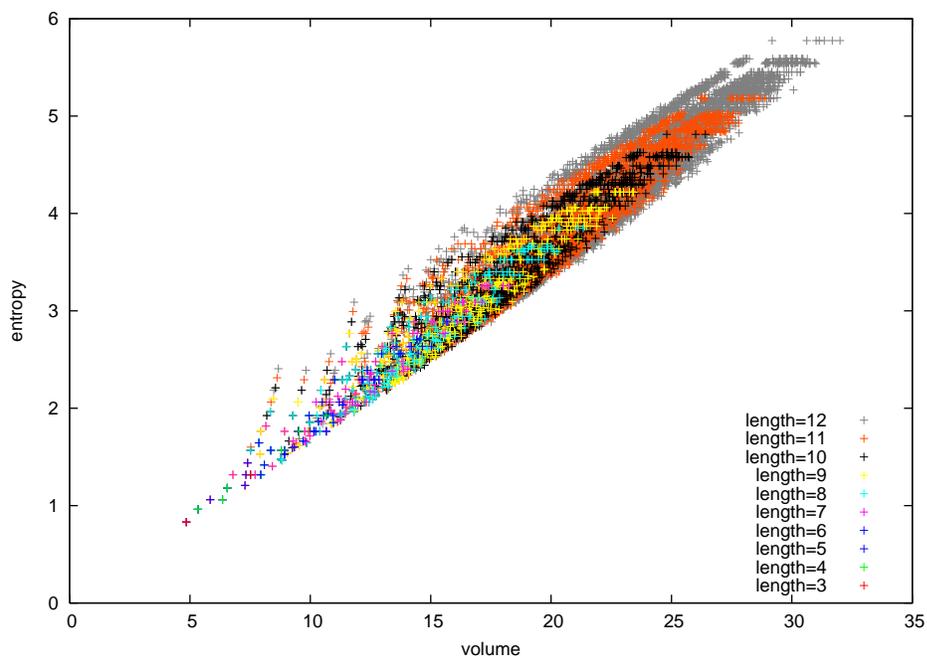}
\caption{$\mathcal{E}_{12}(D_4)$.} 
\label{fig_br4}
\end{center}
\end{figure}

\begin{figure}[htbp]
\begin{center}
\includegraphics[height=0.9\textwidth, angle=-90]{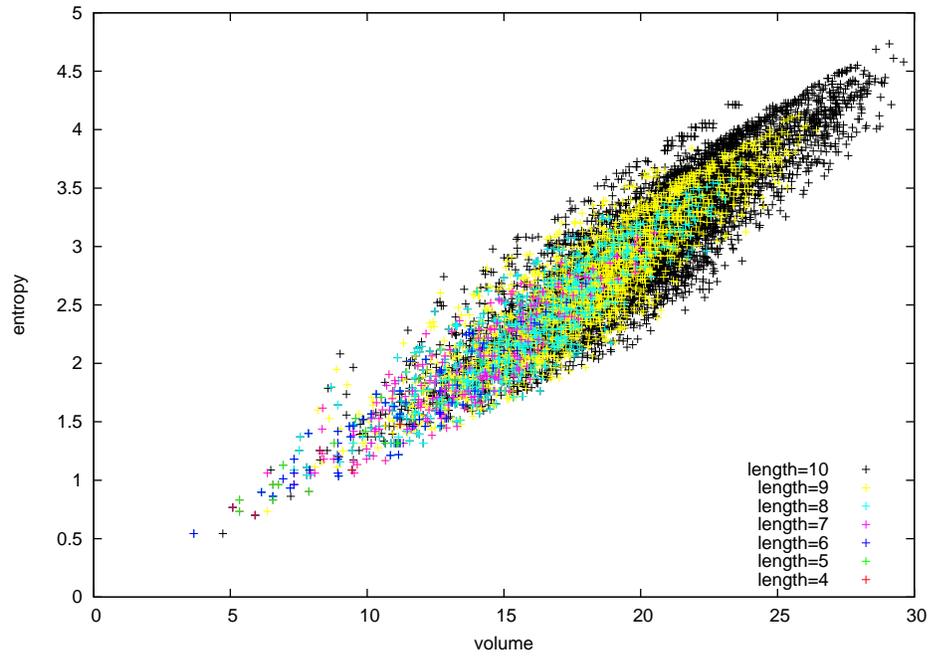}
\caption{$\mathcal{E}_{10}(D_5)$.} 
\label{fig_br5}
\end{center}
\end{figure}

\begin{figure}[htbp]
\begin{center}
\includegraphics[height=0.9\textwidth, angle=-90]{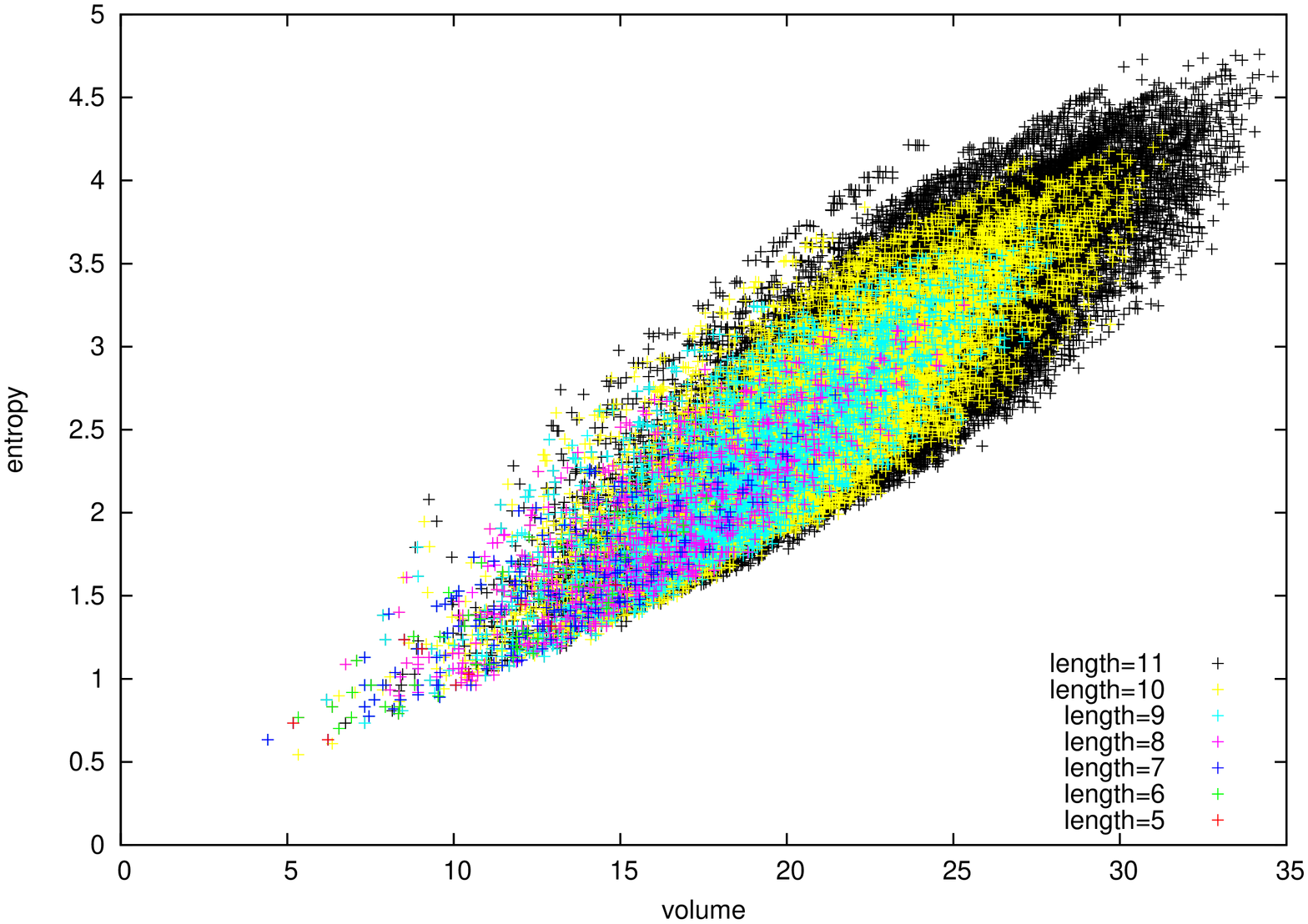}
\caption{$\mathcal{E}_{11}(D_6)$.} 
\label{fig_br6}
\end{center}
\end{figure}


\subsection{Observations}\label{subsection_observations}

Recall that $\lambda(\Sigma)$ is the minimal dilatation among $\lambda(\phi)$ for $\phi \in \mathcal{M}^{\mathrm{pA}}(\Sigma)$. 
We introduce the following notation: 
\begin{eqnarray*}
	\lambda_k(\Sigma)
	&=& \min\{\lambda(\phi) \ |\ 
	\phi \in \mathcal{M}^{\mathrm pA}(\Sigma)\ \mbox{up\ to\ word\ length\ }k\}, \\
	\lambda(\Sigma;c) 
	&=& \min\{\lambda(\phi) \ |\  
	\phi \in \mathcal{M}^{\mathrm pA}(\Sigma),\ 
	{\Bbb T}(\phi)\ \mbox{has\ }c\mbox{\ cusps}\}, \\
	\lambda_k(\Sigma;c) 
	&=& \min\{\lambda(\phi) \ |\ 
	\phi \in \mathcal{M}^{\mathrm pA}(\Sigma)\ \mbox{up\ to\ word\ length\ }k,
	\  {\Bbb T}(\phi)\ \mbox{has\ }c\mbox{\ cusps}\}. 
\end{eqnarray*}
In the case of $\Sigma= D_n$,  the number of the cusps of the mapping torus ${\Bbb T}(b)$ of 
$b \in \mathcal{M}^{\mathrm pA}(D_n) < \mathcal{M}^{\mathrm pA}(\Sigma_{0,n+1})$ 
equals the number of the components of the link $\overline{b}$, 
since ${\Bbb T}(b)= S^3 \setminus \overline{b}$.

The minimal dilatation $\lambda(\Sigma)$, and  the minimal entropy $\mathrm{ent}(\Sigma)$, 
are known for the following surfaces.

\label{equation_mini-dil}
\begin{center}
\begin{tabular}{|  l   |c| c | c | c| c|} \hline
    $\Sigma$                        &  $\lambda(\Sigma)$  & $\mathrm{ent}(\Sigma) $     &mapping  class realizing $\lambda(\Sigma)$  & reference \\ \hline
    $\Sigma_{1,1}$        & $\approx  2.61803$     & $\approx 0.96242$    &$t_a t_b^{-1}$ &  folklore  \\
   $D_3$                     &   $\approx  2.61803$    & $\approx 0.96242$      & $\beta_3:= \sigma_1 \sigma_2^{-1}$ & Matsuoka \cite{Matsuoka} \\
   $D_4$                     &  $ \approx  2.29663$    & $\approx 0.83144$      & $\beta_4:=\sigma_1 \sigma_2 \sigma_{3}^{-1}$ & Ko-Los-Song \cite{KLS} \\
   $D_5$                     &   $\approx 1.72208 $   & $\approx 0.54353$      &$ \beta_5:=\sigma_1^3 \sigma_2 \sigma_3 \sigma_4$ &Ham-Song \cite{HS}  \\ 
   $\Sigma_{2,0}$                     &   $\approx 1.72208 $   &
	 $\approx 0.54353$      &
$ \tilde{\beta_5}:= t_a^3 t_b t_c t_d$
 &Cho-Ham \cite{CH}  \\ \hline
\end{tabular}
\end{center}

We now turn to the volume. 
The set 
$$\{v >0\ |\ v\ \mbox{is\ the\ volume\ of\ a\ hyperbolic\ }3\mbox{-manifold}\},$$ 
called the {\it volume spectrum}, is a well-ordered closed subset 
of the set of real numbers ${\Bbb R}$ of order type $\omega^{\omega}$. 
In particular any subset of the volume spectrum achieves its infimum. 
We set 
\begin{eqnarray*}
	\mathrm{vol}(\Sigma) 
	&=& \min \{\mathrm{vol}(\phi)\ |\ 
	\phi \in \mathcal{M}^{\mathrm{pA}}(\Sigma)\}, \\
	\mathrm{vol}_k(\Sigma)
	&=& \min\{\mathrm{vol}(\phi)\ |\  
	\phi \in \mathcal{M}^{\mathrm{pA}}(\Sigma)\ \mbox{up\ to\ word\ length\ }k \}, \\
	\mathrm{vol}(\Sigma;c) 
	&=& \min\{\mathrm{vol}(\phi)\ |\  \phi \in \mathcal{M}^{\mathrm{pA}}(\Sigma),\  
	{\Bbb T}(\phi)\ \mbox{has\ }c\mbox{\ cusps}\},\  \mbox{and}\\
	\mathrm{vol}_k(\Sigma;c) 
	&=& \min\{\mathrm{vol}(\phi)\ |\  \phi \in
	\mathcal{M}^{\mathrm{pA}}(\Sigma)\ \mbox{up\ to\ word\ length\
	}k, \ 
	 {\Bbb T}(\phi)\ \mbox{has\ }c\mbox{\ cusps}\}. 
\end{eqnarray*}

A question is which mapping class
reaches $\lambda(\Sigma)$, 
and 
which one reaches $\mathrm{vol}(\Sigma)$. 
For the case $\Sigma \in \{D_3, D_5\}$, 
there exists a mapping class
simultaneously reaching both $\lambda(\Sigma)$ and
$\mathrm{vol}(\Sigma)$.
The $3$-braid $\beta_3$ with minimal dilatation realizes 
$\mathrm{vol}(D_3)$ (F.~Gu\'eritaud and D.~Futer \cite[Theorem~B.1]{GF}). 
For the $5$-braid $\beta_5$ with minimal dilatation, 
the link $\overline{\beta_5}$ equals the $(-2,3,8)$-pretzel link (Figure~\ref{fig_pretzel}). 
It is shown that the $(-2,3,8)$-pretzel link complement and 
the Whitehead link complement  have the minimal volume 
among orientable $2$-cusped hyperbolic $3$-manifolds (Agol \cite{Agol2}). 
Hence $\beta_5$ also realizes $\mathrm{vol}(D_5)\approx 3.66339$. 
\begin{figure}[htbp]
\begin{minipage}{0.5\textwidth}
\begin{center}
\includegraphics[width=2.5in]{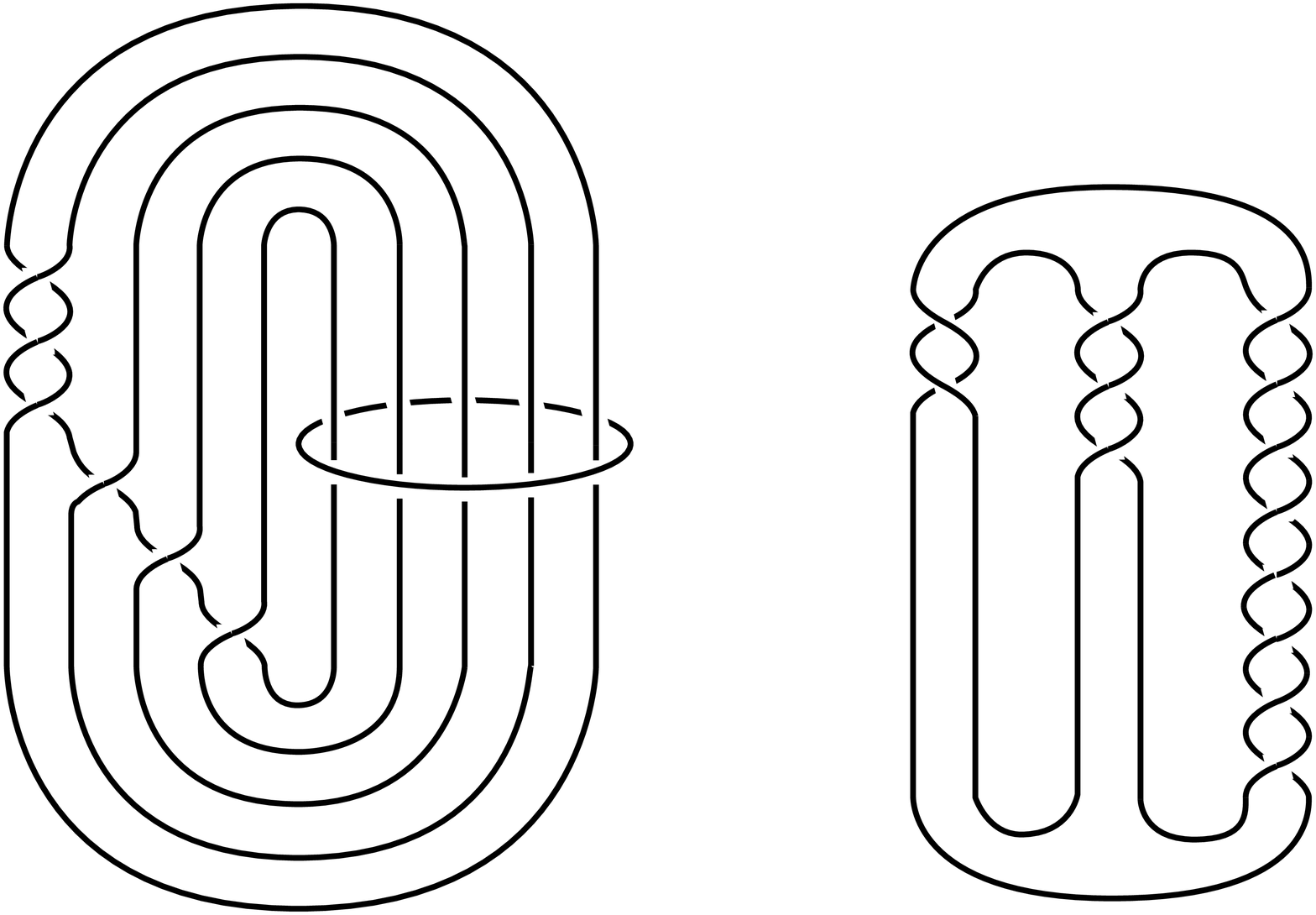}
\caption{link $\overline{\beta_{5}}$ (left) is equal to  $(-2,3,8)$-pretzel link (right).} 
\label{fig_pretzel}
\end{center}
\end{minipage}
\hfill
\begin{minipage}{0.48\textwidth}
\begin{center}
\includegraphics[width=1.6in]{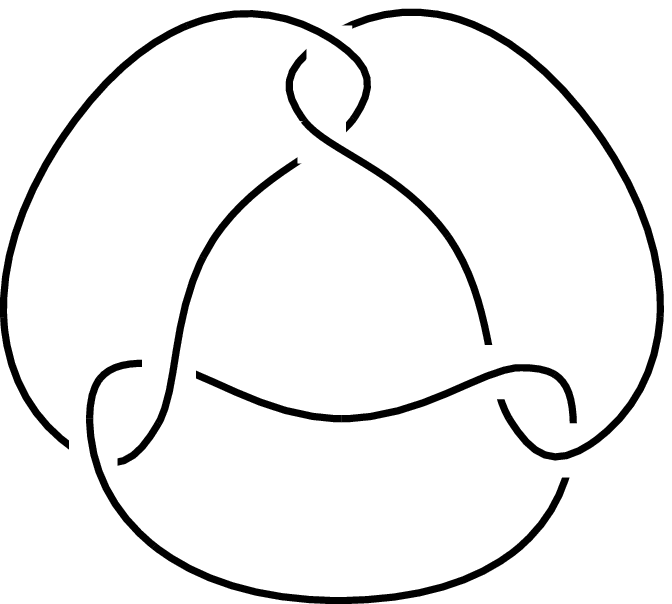}
\caption{chain-link with $3$ components.} 
\label{fig_chain_link_rev}
\end{center}
\end{minipage}
\end{figure} 
One may ask that whether there exists a mapping class simultaneously reaching
both $\lambda(\Sigma)$ and $\mathrm{vol}(\Sigma)$.
This question seems to be false in general.
In our experiment, for the case $D_6$,
$\mathrm{vol}_{11}(D_6)$ and 
$\lambda_{11}(D_6)$ are
not reached by the same mapping class.
It may be caused by the fact that 
the mapping torus reaching $\mathrm{vol}_{11}(D_6)$ and 
the one reaching $\lambda_{11}(D_6)$
have different number of cusps.
It would be natural to revise the question by:

\begin{ques}
Does there exist a mapping class of $\mathcal{M}^{\mathrm{pA}}(\Sigma)$ simultaneously reaching both 
$\lambda(\Sigma;c)$ and $\mathrm{vol}(\Sigma;c)$?
\end{ques}

\noindent
The computation together with theoretical results shows this is likely to be positive.
More precisely: 

\begin{itemize}
\item[{\bf 1-a.}]
The $3$-braid $\beta_3$ reaches 
both $\lambda(D_3) \approx 2.61803$ and
$\mathrm{vol}(D_3) \approx 4.05976$.
Thus
$\lambda(D_3) = \lambda(D_3; 2)$ and
$\mathrm{vol}(D_3) = \mathrm{vol}(D_3; 2)$.
\item[{\bf 1-b.}]
	It is easy to verify that the $3$-braid $\sigma_1^2 \sigma_2^{-1}$ 
	reaches $\lambda(D_3;3)$. This 3-braid also reaches $\mathrm{vol}_{15}(D_3;3) =
	 \mathrm{vol}(S^3 \setminus C_3) \approx 5.33348$, 
	where $C_3$ is the chain-link with $3$ components (Figure~\ref{fig_chain_link_rev}). 
	Among orientable $3$-cusped hyperbolic $3$-manifolds,  $S^3 \setminus C_3$, 
	which is called the magic manifold, is the one with the smallest known volume. 
\item[{\bf 2-a.}]
	The $4$-braid $\beta_4$ reaches both $\lambda(D_4)=
	     \lambda(D_4;2) \approx 2.29663 $ and   
	$\mathrm{vol}_{12}(D_4)\approx  4.85117 $. 
\item[{\bf 2-b.}]
	The $4$-braid $\sigma_1^2 \sigma_2 \sigma_3^{-1}$ reaches 
	both $\lambda_{12}(D_4;3) \approx 2.61803$ and 
	$\mathrm{vol}_{12}(D_4;3)= \mathrm{vol}(S^3 \setminus C_3)$. 
\item[{\bf 3-a.}]
	The $5$-braid $\beta_5$
	reaches both $\lambda(D_5) = \lambda(D_5;2) \approx 1.72208$ and   
	$\mathrm{vol}(D_5) = \mathrm{vol}(D_5;2) \approx 3.66386$. 
\item[{\bf 3-b.}]
	The $5$-braid $\sigma_1 \sigma_2^2 \sigma_3 \sigma_4 $ reaches 
	both $\lambda_{10}(D_5;3) \approx 2.08102$ and 
	$\mathrm{vol}_{10}(D_5;3) = \mathrm{vol}(S^3 \setminus C_3)$. 
\item[{\bf 4-a.}]
	The $6$-braid 
	$\sigma_1^3\sigma_2 \sigma_3 \sigma_4 \sigma_5 $  reaches both 
	$\lambda_{11}(D_6;2) \approx 1.8832$ and $  \mathrm{vol}_{11}(D_6;2) \approx 4.41533 $. 
\item[{\bf 4-b.}] 
	The $6$-braid 
 	$ \sigma_1^3 \sigma_2 \sigma_1^2 \sigma_3 \sigma_2 \sigma_4
 	\sigma_5 $ 
	reaches both  $\lambda_{11}(D_6;3) = \lambda(\beta_5) $ and  
	$\mathrm{vol}_{11}(D_6;3) = \mathrm{vol}(S^3 \setminus C_3)$. 
\end{itemize}


These observations are straightforward from
the following plots of $\mathcal{E}_{15}(D_3)$, 
$\mathcal{E}_{12}(D_4)$, $\mathcal{E}_{10}(D_5)$ and
$\mathcal{E}_{11}(D_6)$, restricted to
the range of the volume $< 5.334$
(Figure \ref{fig_smallvolume}).

We can check that in the plots 
all the 3-cusped mapping tori have the same volume $ \approx 5.33348 $ and 
other mapping tori have 2 cusps and smaller volumes than the one with 3 cusps.

\begin{figure}[htbp]
\scalebox{0.5}{\input{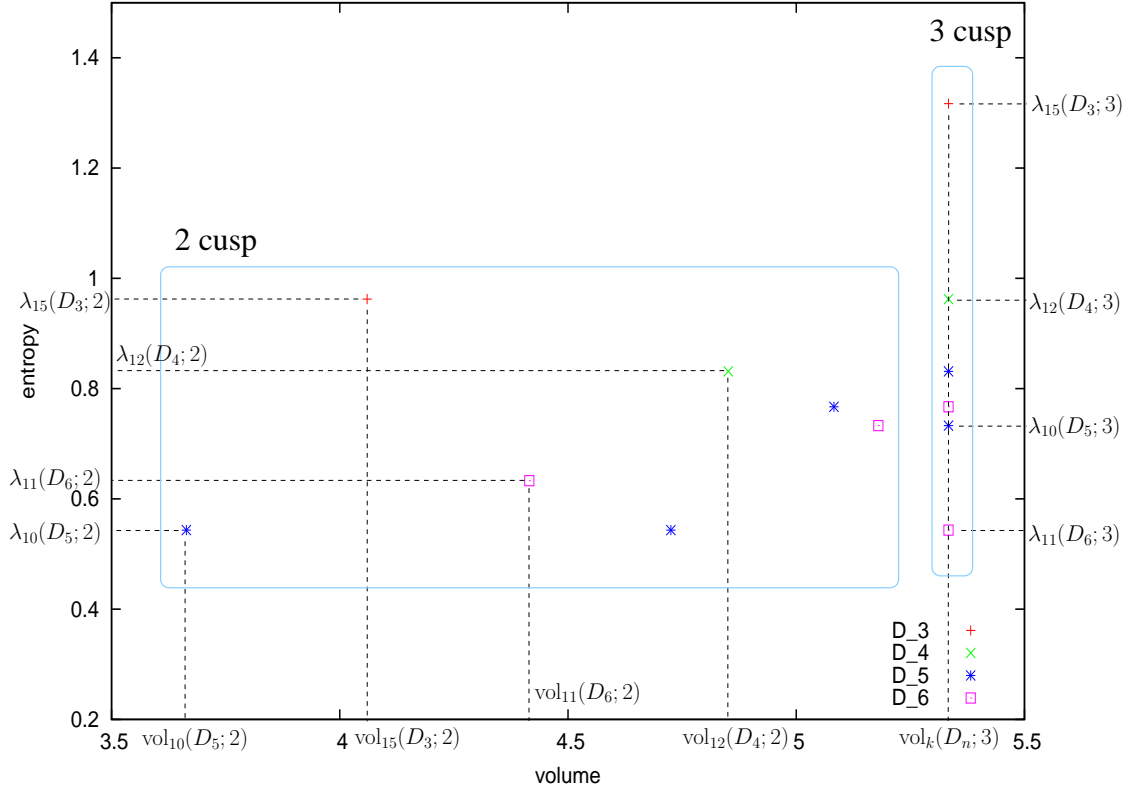}}
\caption{$\mathcal{E}_{15}(D_3)$, 
$\mathcal{E}_{12}(D_4)$, $\mathcal{E}_{10}(D_5)$ and
$\mathcal{E}_{11}(D_6)$.}
\label{fig_smallvolume}
\end{figure}


\begin{rem}
In the experiment,
all braids in (1-b), (2-b), (3-b) and (4-b) have the same volume
 $ \approx 5.33348 \approx \mathrm{vol}(S^3 \setminus C_3)  $.
Actually we can verify these mapping tori are homeomorphic to $S^3
 \setminus C_3$, see \cite{KT}.
\end{rem}

\begin{rem}
\noindent
	The $5$-braids $\beta_5$ and 
	$\beta_5'= \sigma_1^4 \sigma_2 \sigma_3 \sigma_1 \sigma_2 \sigma_3 \sigma_4$ 
	both realize the minimal dilatation $\lambda(D_5)$, but the inequality 
	$\mathrm{vol}(D_5) = \mathrm{vol}(\beta_5) < \mathrm{vol}(\beta_5')$ holds. 
	This example says that it is not true that the mapping class with minimal dilatation 
	realizes the minimal volume. 
\end{rem}

In Theorem \ref{thm_Bi-Lipschitz}
it is shown that there exists a constant
$B = B(\Sigma)$ such that $B \ {\rm vol}(\phi) \leq {\rm ent}(\phi)$ holds.
However it is not quite obvious to find accurate value of $B$.
We set 
\begin{eqnarray*} 
\mathcal{I}(\Sigma) &=& \inf
 \Big\{\frac{\mathrm{ent}(\phi)}{\mathrm{vol}(\phi)}\ \Big|\ \phi \in
 \mathcal{M}^{\mathrm{pA}}(\Sigma) \Big\}\  \mbox{and} \\
\mathcal{I}_k(\Sigma) &=& \min
\Big\{\frac{\mathrm{ent}(\phi)}{\mathrm{vol}(\phi)}\ \Big|\ \phi \in
\mathcal{M}^{\mathrm{pA}}(\Sigma)  \ \mbox{up\  to\  word\  length\ } k \Big\}.
\end{eqnarray*}


\begin{figure}[htbp]
\begin{center}
\includegraphics[height=\textwidth, angle=-90]{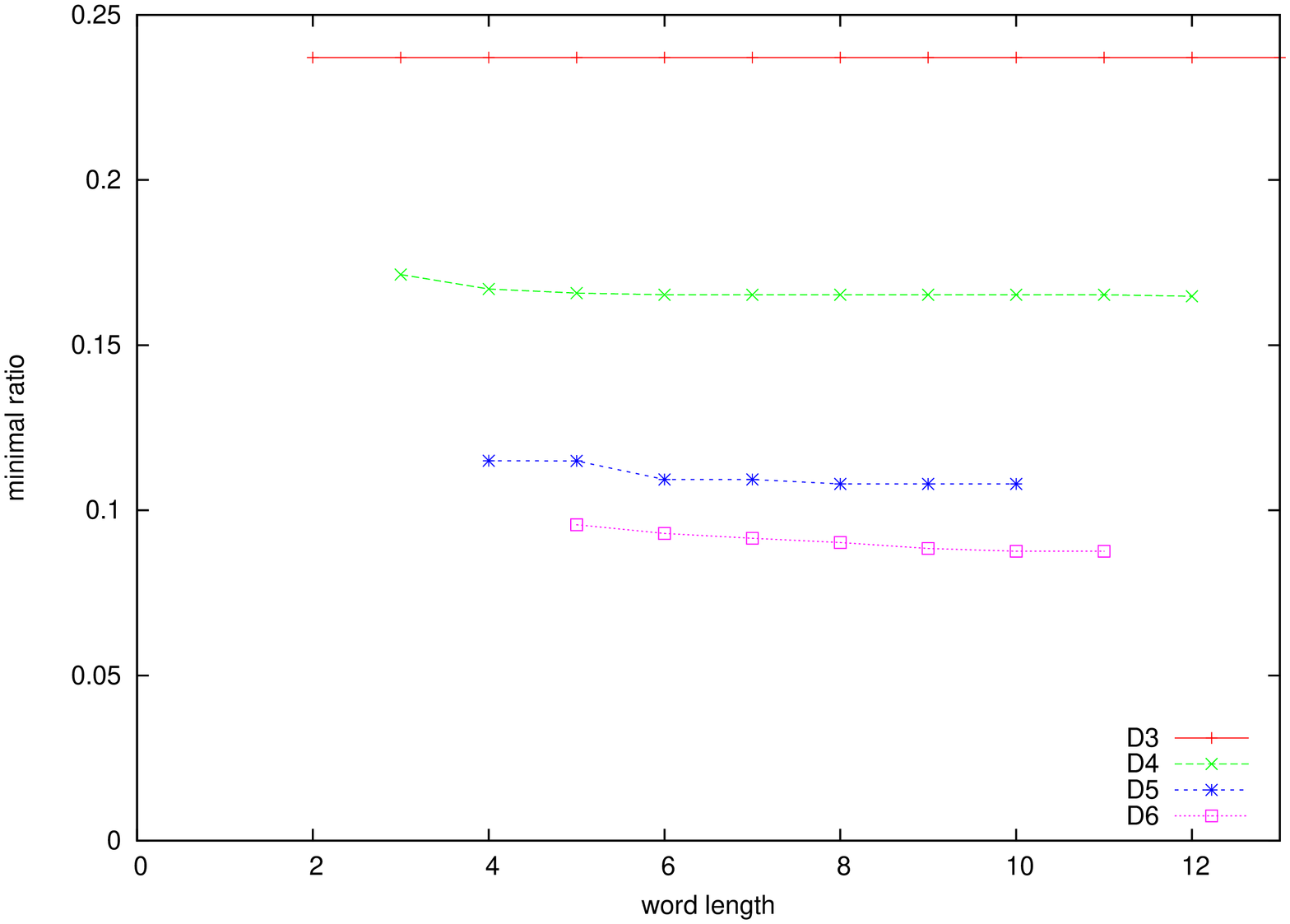}
\caption{minimal ratio up to some word length for $\Sigma \in \{D_3,D_4, D_5, D_6\}$.}
\label{fig_ratio}
\end{center}
\end{figure}

\noindent
It is natural to ask:
\begin{ques}
\label{ques_attain}
Does  there exist a mapping class  $\phi \in  \mathcal{M}^{\mathrm{pA}}(\Sigma)$ which attains
$\mathcal{I}(\Sigma)$ ? 
\end{ques}

To study Question~\ref{ques_attain}, we compute ${\mathcal I}_k(\Sigma)$
for 
$\Sigma \in \{D_3,D_4,D_5,D_6\}$ (Figure~\ref{fig_ratio}). 
We see that ${\mathcal I}_k(D_3)$  is achieved by the mapping class $\sigma_1 \sigma_2^{-1}$ up to $k=15$. 
On the other hand for any other surfaces,  
${\mathcal I}_k(\Sigma)$ decreases  as $k$ increases. 
We will study
${\mathcal I}(D_3)$ and ${\mathcal I}(\Sigma_{1,1})$ 
in the next section.


\begin{ques}
Is it true that $ \mathcal{I}(D_n)>  \mathcal{I}(D_{n+1})$ for all $n \ge 3$?
Is it true that $ \mathcal{I}(\Sigma_{g,0})>  \mathcal{I}(\Sigma_{g+1,0})$ for all $g \ge 2$? 
\end{ques}

%
%

\section{A lower bound of $\mathcal{I}(\Sigma_{1,1})$}
\label{section_a-lower-bound}

Let $a$ and $b$ be  the meridian and the longitude of a once-punctured torus. 
We set $L= t_a$ and $R = t_b^{-1}$, noting that $\{a, b\}$ fills $\Sigma_{1,1}$. 
We first recall the following well known result: 

\begin{lem}
\label{thm_canonical-form}
Each $\phi \in \mathcal{M}^{\mathrm{pA}}(\Sigma_{1,1})$ is conjugate to a mapping class 
\begin{eqnarray}
\label{expression_conjugacy-class}
L^{m_1} R^{n_1} \cdots L^{m_{\ell}} R^{n_{\ell}} \in \mathcal{R}(\{a\}, \{b\})
\end{eqnarray}
where $\ell$, $m_i$ and $n_i$ are positive integers, 
and the mapping class $L^{m_1} R^{n_1} \cdots L^{m_{\ell}} R^{n_{\ell}}$ 
is unique up to cyclic permutations. 
Conversely, every  mapping class of the form 
{\em (\ref{expression_conjugacy-class})} is pseudo-Anosov. 
\end{lem}
\noindent 
The integer $\ell$ in the above is called the {\it block length} of $\phi$. 
\medskip

\begin{thm}
\label{thm_punctured-torus}
For each $\phi \in  \mathcal{M}^{\mathrm{pA}}(\Sigma_{1,1})$, we have 
\begin{equation*} 
	\frac{\mathrm{ent}(\phi)}{\mathrm{vol}(\phi)} > 
	\frac{ \log (\frac{3+  \sqrt{5}}{2})}{2  v_8}\approx 0.1313,
\end{equation*} 
where $v_8 \approx 3.6638$ is the volume of a regular ideal octahedron. 
\end{thm} 

\begin{proof}  
Let $M_L = \left(\begin{array}{cc}1 & 1 \\0 & 1\end{array}\right)$ and 
$M_R = \left(\begin{array}{cc}1 & 0 \\1 & 1\end{array}\right)$.  
The matrix  $M_L$ is the incident matrix for $L$ and the matrix $M_R$ is 
the incident matrix for $R$. 
Suppose that $\phi = L^{m_1} R^{n_1} \cdots L^{m_{\ell}} R^{n_{\ell}}$ of block length $\ell$. 
Then 
\begin{equation*} 
	M'= M_L^{m_1} M_R^{n_1} M_L^{m_2} M_R^{n_2} \cdots 
	M_L^{m_{\ell}} M_R^{n_{\ell}} 
\end{equation*} 
is the incident matrix for $\phi$. 
Since  $M' \ge (M_L M_R)^{\ell}$, 
the largest eigenvalue of $M'$ is greater than that of $(M_L M_R)^{\ell}$. 
We thus have 
\begin{equation*} 
	\lambda(\phi) \ge \lambda ((LR)^{\ell} ) = \Bigl(\frac{3 + \sqrt{5}}{2} \Bigr)^{\ell}.
\end{equation*}

On the other hand, using a result \cite[Corollary~2.4]{Agol} of Agol, 
\begin{equation}
\label{equation_Volume-Estimate}
\mathrm{vol}(\phi) < 2 \ell v_8. 
\end{equation}
Hence we have 
$$
\frac{\mathrm{ent}(\phi)}{ \mathrm{vol}(\phi)} > \frac{\ell \cdot \log (\frac{3+  \sqrt{5}}{2})}{2 \ell v_8} = 
\frac{ \log (\frac{3+  \sqrt{5}}{2})}{2  v_8}\approx 0.1313.  \
$$
\end{proof}

With the aid of SnapPea, one can show:

\begin{prop} 
\label{prop_block-length-1}
For each $\phi \in \mathcal{M}^{\mathrm{pA}}(\Sigma_{1,1})$ of block length $1$  
\begin{equation*} 
	\frac{\mathrm{ent}(\phi)}{\mathrm{vol}(\phi)} \ge   
	\frac{\mathrm{ent}(LR)}{ \mathrm{vol}(LR)} =  \frac{ \log (\frac{3+  \sqrt{5}}{2})}{2  v_3} 
	\approx 0.4741.
\end{equation*}
\end{prop} 

\begin{proof} 
Let $c_{1,1}^1= \displaystyle{\frac{\mathrm{ent}(LR)}{ \mathrm{vol}(LR)}}$. 
If $\mathrm{ent}(\phi) \ge c_{1,1}^1 \cdot 2 v_8$, 
then $ \displaystyle{\frac{\mathrm{ent}(\phi)}{ \mathrm{vol}(\phi)}} > c_{1,1}^1$ 
since $\mathrm{vol}(\phi) < 2 v_8$ 
(see (\ref{equation_Volume-Estimate})).  
Let 
$$Y= \{ \phi \in \mathcal{M}^{\mathrm{pA}}(\Sigma_{1,1}) \ \mbox{of\ block\ length\ }1\ |\ \mathrm{ent}(\phi) <  c_{1,1}^1 \cdot 2 v_8 < 3.4748\}.$$
This is a finite set. 

If $\phi$ is written as $L^m R^n$, then $\lambda(\phi)$ is 
the largest eigenvalue of $\left(\begin{array}{cc}1+mn & m \\n & 1\end{array}\right)$. 
Thus 
\begin{equation*} 
	\lambda(\phi)= \frac{2+mn+ \sqrt{4mn+(mn)^2}}{2}.
\end{equation*} 
If $L^m R^n \in Y$, then 
$\displaystyle{\frac{2+mn+ \sqrt{4mn+(mn)^2}}{2}} < e^{3.4748} < 33$. 
Hence we have $mn \le 31$. 
We compute 
$\displaystyle{\frac{\mathrm{ent}(\phi)}{ \mathrm{vol}(\phi)}}$ 
for each $\phi = L^m R^n$ with $mn \le 31$ by SnapPea, and we see that 
 it is greater than or  equal to $ c_{1,1}^1$. 
\end{proof}

In the case of $\Sigma_{1,1}$ and $D_3$, we propose the following conjectures: 

\begin{conj}
$$\mathcal{I}(\Sigma_{1,1})= \frac{\mathrm{ent}(LR)}{\mathrm{vol}(LR)}
\approx 0.4741\ \hspace{2mm} \mbox{and}\  \hspace{2mm}
\mathcal{I}(D_3)= \frac{\mathrm{ent}(\sigma_1 \sigma_2^{-1})}{\mathrm{vol}(\sigma_1 \sigma_2^{-1})} 
\approx 0.2370.$$
\end{conj}

%
%


%
%

\end{document}